\documentclass[12pt]{article}
\usepackage{amssymb,amsmath,amsfonts}
\usepackage{amsthm}
\usepackage{txfonts}
\usepackage{color}
\usepackage{epsfig}

\numberwithin{equation}{section} \theoremstyle{plain}
\newtheorem{theorem}{Theorem}[section]

\newtheorem{lemma}[theorem]{Lemma}

\theoremstyle{definition}

\newtheorem{example}{Example}[section]
\theoremstyle{remark}
\newtheorem{remark}{\rm\bf Remark}[section]

\allowdisplaybreaks
\newcommand{\qbinom}[2]{\left[\genfrac{}{}{0pt}{}{#1}{#2}\right]_{q}}

\begin{document}

\title{\bf On the eigenstructure of the $(\alpha,q)$-Bernstein operator}
\author{B\"ulent K\"oro\u{g}lu$^1\footnote{Corresponding author}$ \ and Fatma Ta\c{s}delen Ye\c{s}ildal$^2$}
\date{}
\maketitle

\begin{center}
{\it $^1$Social Security Institution, Data Center, Yenimahalle, Ankara, Turkey \\ e-mail: koroglubulent3@gmail.com}\\
{\it $^2$Department of Mathematics, Ankara University 06100, Beşevler, Ankara, Turkey \\ e-mail: tasdelen@science.ankara.edu.tr}\\
\end{center}

\begin{abstract}
We obtain eigenvalues and eigenvectors of the $(\alpha,q)$-Bernstein operator $T_{n,q,\alpha}$. Moreover, we will give the limit behaviour of these eigenvalues and eigenvectors for all $q.$ 
\end{abstract}


%{\bf 2010 MSC:}  

\section{Introduction}
The well-known Weierstrass Approximation Theorem, proved by Karl Weierstrass in 1885, states that for any continuous function $f$ defined in interval $[a,b]$ and $\epsilon>0$, there exists a polynomial $P$ such that $|f(x)-P(x)|<\epsilon$. Since the proof of the theorem is lengthy and complicated, many researchers studied to find simple and effective proof. In 1912, using probability theory, S.N Bernstein developed Bernstein polynomials and published rather simple and effective proof for Weierstrass Approximation Theorem \cite{Bern}. In 1987, Lupa\c{s} has defined $q$-analogue of the Bernstein polynomials for rational functions \cite{lupas}. However, in 1997, Phillips has developed $q$-analogue of polynomial functions such that for $q=1,$ classical Bernstein operators are obtained \cite{phil}. After this date, many researchers studied $q$-Bernstein operators and with the help of this operator, $q$-analogues of many other operators are obtained 
\cite{jat,KR67,novikov,jat2,decade,parametric,wangnorms}.

Many researchers have studied the eigenstructure of the extensions of Bernstein and $q$-Bernstein operators. See, for example, \cite{goska2, goska, margareta,  kernel, wng}. The present work is based on the studies by S. Cooper and S. Waldron in 2000 related to the eigenstructure of the classical and multi-variable Bernstein operators \cite{waldron, waldron2}. Eigenfunctions of the $q$-Bernstein operators and their asymptotic behavior are investigated by S. Ostrovska and M.Turan in 2013 \cite{memocan}. In this study, eigenvalues and eigenfunctions of the $(\alpha,q)$-Bernstein operators introduced by Qing-Bo Cai and Xiao-Wei Xu in \cite{qingbo} are found and their asymptotic behavior is investigated. When $q=1,$ one obtains the eigenvalues and eigenfunctions of the $\alpha$-Bernstein operator introduced by Chen et. al \cite{xia}.

\section{Preliminaries}

The definitions and notations used in this article are adopted from \cite[Ch.10]{askey}.
Let $q>0.$ The $q$-integer is defined by

\begin{equation}\label{w1}
[n]_q:=1+q+q^2+\dots+q^{n-1},\quad [0]_q:=0 \qquad(n=1,2,\dots),
\end{equation} 
the $q$-factorial of $n$ by
\begin{align*}
[n]_q!:=[1]_q[2]_q\dots [n]_q,\quad[0]_q!:=1 \qquad(n=1,2,\dots).
\end{align*} 
From \eqref{w1}, one can easily see that 
\begin{eqnarray}\label{prop1}
[n-i]_q=\frac{[n]_q-[i]_q}{q^i}, \quad i=0,1,\ldots,n,
\end{eqnarray}
For integers $k$ and $n$ with $0\leq k\leq n$, the $q-$ binomial coefficient is 
\begin{equation*}
\qbinom{n}{k}=\frac{[n]_q!}{[k]_q![n-k]_q!}
\end{equation*}
The $q$-shifted products defined by
$$(a;q)_0=1, \quad (a;q)_k=\prod_{s=0}^{k-1}(1-aq^s),\quad (a;q)_\infty=\prod_{s=0}^{\infty}(1-aq^s).$$
We also need the $q$-Stirling numbers of the second kind 
$S_q(k,r)$ given by
\begin{align}\label{str1}
S_q(k,r)=\frac{1}{[r]_q!\, q^{r(r-1)/2}}\sum_{i=0}^r(-1)^iq^{\frac{i(i-1)}{2}}\qbinom{r}{i}[r-i]_q^k.
\end{align}
Using the induction on $k$, one can verify that 
\begin{align}\label{str2}
S_q(k+1,r)=S_q(k,r-1)+[r]_qS_q(k,r), \qquad k\geq0,\: r\geq1
\end{align}
with $S_q(0,0)=1$, $S_q(k,0)=0$ for $k>0$ and $S_q(k,r)=0$ for $k<r.$

The $\alpha$-Bernstein polynomial of $f:[0,1]\rightarrow R$ is introduced in \cite{xia} as
\begin{equation}\label{tnf}
T_{n,\alpha}(f;x)=\sum_{i=0}^n\,f\left(\frac{i}{n}\right)p_{n,i}^{(\alpha)}(x)
\end{equation}
where  $p_{n,i}^{(\alpha)}(x)$ are the $\alpha$-Bernstein polynomials of degree $n$ given by
$p_{1,0}^{(\alpha)}(x)=1-x$, $p_{1,1}^{(\alpha)}(x)=x$ and for $n\geq 2$,
\begin{align*}
p_{n,i}^{(\alpha)}(x)&=\left[\binom{n-2}{i}(1-\alpha)x+\binom{n-2}{i-2}(1-\alpha)(1-x)+\binom{n}{i}\alpha x(1-x)\right]x^{i-1}(1-x)^{n-1-i}.
\end{align*}
Here $\binom{n}{i}$ stands for the binomial coefficients and one has $\binom{n}{i}=0$ for $i<0$ or $i>n.$

The $\alpha$-Bernstein operator $T_{n,\alpha}$ on $C[0,1]$ is given by $$ T_{n,\alpha}:f \rightarrow T_{n,\alpha}f $$
Using the forward difference operator, \eqref{tnf} can also be written as 

\begin{equation*}
T_{n,\alpha}(f;x)=\sum_{r=0}^n\left[(1-\alpha)\binom{n-1}{r}\Delta^rg_0+\alpha\binom{n}{r}\Delta^rf_0\right]x^r
\end{equation*}
where 
$$
f_i=f\left(\frac{i}{n}\right), \quad g_i=\left(1-\frac{i}{n-1}\right)f_i+\frac{i}{n-1}f_{i+1}
$$
and $\Delta^{0}f_i=f_i$, $\Delta^{r}f_i=\Delta^{r-1}f_{i+1}-\Delta^{r-1}f_i$ for $r\geq 1.$

Lemma 3.1 in \cite{xia} states that the higher-order forward difference of $g_i$ can be expressed as 
\begin{equation*}
\Delta^rg_i=\left(1-\frac{i}{n-1}\right)\Delta^rf_i+\frac{i+r}{n-1}\Delta^rf_{i+1},
\end{equation*}
The $q$-analogue of $\alpha$-Bernstein operators, called $(\alpha,q)$-Bernstein operators are defined in \cite{qingbo} as 
$T_{n,q,\alpha}:f\rightarrow T_{n,q,\alpha}(f;.)$ such that
$$
T_{n,q,\alpha}(f;x)=\sum_{i=0}^n\,f\left(\frac{[i]_q}{[n]_q}\right)p_{n,q,i}^{(\alpha)}(x),
$$
where $p_{n,q,i}^{(\alpha)}(x)$ are the basis $(\alpha,q)$-Bernstein polynomials of degree $n$ given by 
$p_{1,q,0}^{(\alpha)}(x)=1-x$, $p_{1,q,1}^{(\alpha)}(x)= x$ and 
\begin{align*}
p_{n,q,i}^{(\alpha)}(x)=\bigg(\qbinom{n-2}{i}\frac{(1-\alpha)x}{1-q^{n-i-1}x}+\qbinom{n-2}{i-2}(1-\alpha)q^{n-i-2}
+\qbinom{n}{i}\alpha x\bigg)x^{i-1}(x;q)_{n-i}.
\end{align*}
Like $\alpha$-Bernstein polynomials, the $(\alpha,q)$-Bernstein polynomials have the representation using the forward difference operators as 
\begin{equation}\label{Tnqo}
T_{n,q,\alpha}(f;x)=\sum_{r=0}^n\bigg((1-\alpha)\qbinom{n-1}{r}\Delta_q^rg_0+\alpha\qbinom{n}{r}\Delta_q^rf_0\bigg)x^r
\end{equation}
where 
$\Delta_q^{0}f_i=f_i,$ $\Delta_q^{r}f_i=\Delta_q^{r-1}f_{i+1}-q^{r-1}\Delta_q^{r-1}f_i$ for $r\geq 1.$
Also, the higher-order forward difference of $g_i$ can be expressed as, see \cite[Lemma 2.4]{qingbo},
\begin{equation}\label{delqr}
\Delta^r_qg_i=\left(1-\frac{q^{n-i-1}[i]_q}{[n-1]_q}\right)\Delta^r_qf_i+\frac{q^{n-i-1-r}[i+r]_q}{[n-1]_q}\Delta^r_qf_{i+1}.
\end{equation}
It is worth mentioning that when $\alpha=1,$ the $\alpha$-Bernstein and $(\alpha, q)$-Bernstein polynomials become Bernstein and $q$-Bernstein polynomials, respectively.

\section{Main Results}
Since the $\alpha$-Bernstein polynomials are obtained from $(\alpha,q)$-Bernstein polynomials as $q\to 1,$ we will present only the results for the latter one. Similar results can be derived for $\alpha$-Bernstein polynomials by taking the limit as $q\to 1.$

It is known that $(\alpha,q)$-Bernstein polynomials possess some properties of $q$-Bernstein polynomials.
For example, see \cite{xia} and \cite{qingbo}, they have the end-point interpolation property:
\begin{align*}
T_{n,q,\alpha}(f;0)=f(0), \quad T_{n,q,\alpha}(f;1)=f(1),\quad n=1, 2, \dots, \;q>0,
\end{align*}
and leave the functions invariant:
\begin{equation*}
T_{n,q,\alpha}(at+b;x)=ax+b,\quad n=1,2,\dots, \;q>0.
\end{equation*}
Moreover, they are degree reducing on polynomials, that is, $T_{n,q,\alpha}(t^k;x)$ is a polynomial of degree $\min\{n,k\}.$ This implies that, for $k\leq n$, the subspace ${\mathcal P}_k$ of polynomials of degree at most $k$ is invariant under $T_{n,q,\alpha}.$

To be specific, it is known (see \cite{qingbo}) that 
\begin{eqnarray*}
T_{n,q,\alpha}(t^k;x)=a_{k}x^k+a_{k-1}x^{k-1}+\cdots+a_{1}x,
\end{eqnarray*}
where 
$$
a_{k}=\frac{q^{k(k-1)/2} [n-2]_q!}{[n-k]_q![n]_q^k}\left\{(1-\alpha)[n-k]_q[n-1+k]_q+\alpha[n]_q[n-1]_q\right\}.
$$
The above representation gives only an explicit formula for the leading coefficient. However, in our study, we need all 
coefficients explicitly which are given in the following lemma.

\begin{lemma}
Let $f(t)=t^k.$ Then, for $r\leq k$ and $i=0,1,\ldots,n,$ one has 
\begin{align}\label{lemqdif}
\Delta_q^{r}f_i=\frac{1}{[n]_q^k}\sum_{s=0}^r (-1)^s q^{s(s-1)/2} \qbinom{r}{s} [i+r-s]_q^k. 
\end{align}
\end{lemma}

\begin{proof} Take $f(t)=t^k$ in \cite[formula (2.1)]{phil}.
\end{proof}

\begin{lemma}
For $n\geq k \geq 1$, one has 
\begin{equation*}
T_{n,q,\alpha}(t^k;x)=\sum_{r=0}^{k}a(r,k)x^r
\end{equation*}
where 
\begin{multline}
a(r,k)=\frac{q^{\frac{r(r-1)}{2}}[n-2]_q!}{[n]_q^k[n-r]_q!}\bigg\{(1-\alpha)[n-r]_q\bigg([n+r-1]_qS_q(k+1,r+1) \\
-[r+1]_q[n-1]_qS_q(k,r+1)\bigg)+\alpha[n]_q[n-1]_qS_q(k,r)\bigg\}.\label{ark}
\end{multline}
\end{lemma}

\begin{proof} From \eqref{Tnqo} one has 
\begin{align*}
a(r,k)=(1-\alpha)\qbinom{n-1}{r}\Delta_q^rg_0+\alpha\qbinom{n}{r}\Delta_q^rf_0.
\end{align*}
For $i=0$, \eqref{lemqdif} becomes
\begin{align*}
\Delta_q^{r}f_0=\frac{1}{[n]_q^k}\sum_{s=0}^r (-1)^s q^{s(s-1)/2} \qbinom{r}{s} [r-s]_q^k,
\end{align*}
and using \eqref{str1}, we get 
\begin{align*}
\Delta_q^rf_0=\frac{[r]_q!q^{r(r-1)/2}}{[n]_q!}S_q(k,r).
\end{align*}
As $\Delta_q^{r+1}f_0=\Delta_q^rf_1-q^r\Delta_q^rf_0$, we have 
\begin{align*}
\Delta_q^rf_1&=\Delta_q^{r+1}f_0+q^r\Delta_q^rf_0\\
             &=\frac{[r]_q!q^{r(r+1)/2}}{[n]^k_q}([r+1]_qS_q(k,r+1)+S_q(k,r))\\
						 &=\frac{[r]_q!q^{r(r+1)/2}}{[n]^k_q}S_q(k+1,r+1)
\end{align*}
Also, $i=0$ in \eqref{delqr} results in
\begin{align*}
\Delta_q^rg_0&=\Delta_q^rf_0+\frac{q^{n-1-r}[r]_q}{[n-1]_q}\Delta_q^rf_1\\
						 &=\frac{[r]_q!q^{\frac{r(r-1)}{2}}}{[n]_q^k}\bigg(S_q(k,r)+\frac{q^{n-1}[r]_q}{[n-1]_q}S_q(k+1,r+1)\bigg)
\end{align*}
Therefore, 
\begin{align*}
a(r,k)&=\frac{[r]_q!q^{\frac{r(r-1)}{2}}}{[n]_q^k}\bigg\{(1-\alpha)\qbinom{n-1}{r}\left(S_q(k,r)+\frac{q^{n-1}[r]_q}{[n-1]_q}S_q(k+1,r+1) \right)+\alpha\qbinom{n}{r}S_q(k,r)\bigg\}\\
&=\frac{q^{\frac{r(r-1)}{2}}}{[n]_q^k}\bigg\{(1-\alpha)\frac{[n-1]_q!}{[n-r-1]_q!}\left(S_q(k,r)+\frac{q^{n-1}[r]_q}{[n-1]_q}S_q(k+1,r+1)\right)+\alpha\frac{[n]_q!}{[n-r]_q!}S_q(k,r)\bigg\}.
\end{align*}
Using \eqref{str2}, the last equality becomes 
\begin{align*}
a(r,k)
&=\frac{q^{\frac{r(r-1)}{2}}[n-2]_q!}{[n]_q^k[n-r]_q!}\bigg\{(1-\alpha)[n-1]_q[n-r]_q\bigg(S_q(k+1,r+1)-[r+1]_qS_q(k,r+1)\\
&\qquad+\frac{q^{n-1}[r]_q}{[n-1]_q}S_q(k+1,r+1)\bigg)+\alpha[n]_q[n-1]_qS_q(k,r)\bigg\}\\
&=\frac{q^{\frac{r(r-1)}{2}}[n-2]_q!}{[n]_q^k[n-r]_q!}\bigg\{\bigg((1-\alpha)[n-r]_q\big([n-1]_q+q^{n-1}[r]_q\big)S_q(k+1,r+1)\\
&\qquad -[r+1]_q[n-1]_qS_q(k,r+1)\bigg)+\alpha[n]_q[n-1]_qS_q(k,r)\bigg\}\\
&=\frac{q^{\frac{r(r-1)}{2}}[n-2]_q!}{[n]_q^k[n-r]_q!}\bigg\{(1-\alpha)[n-r]_q\bigg([n+r-1]_qS_q(k+1,r+1)\\
&\qquad-[r+1]_q[n-1]_qS_q(k,r+1)\bigg)+\alpha[n]_q[n-1]_qS_q(k,r)\bigg\}
\end{align*}
which is \eqref{ark} as claimed.
\end{proof}
\begin{lemma}\label{eigdifq}
The numbers 
\begin{align*}
\lambda_{k,q}^{(\alpha,n)}=\frac{q^{\frac{k(k-1)}{2}}[n-2]_q!}{[n-k]_q![n]_q^k}(\left(1-\alpha)[n-k]_q[n-1+k]_q+\alpha[n]_q[n-1]_q\right)
\end{align*}
are distinct for $\alpha\in[0,1]$ and $k=2,\ldots,n$.
\end{lemma}
\begin{proof} One can write 
\begin{eqnarray}\label{43}
\lambda_{k,q}^{(\alpha,n)}=\left(\alpha+(1-\alpha)\frac{[n-k]_q[n+k-1]_q}{[n]_q[n-1]_q}\right)\prod_{m=1}^{k-1}\left(1-\frac{[m]_q}{[n]_q}\right). 
\end{eqnarray}
Dividing \eqref{43} by $\lambda_{k-j,q}^{(\alpha,n)}$ for $j=1,2,\ldots,k-1$, we get
\begin{align*}
\frac{\lambda_{k,q}^{(\alpha,n)}}{\lambda_{k-j,q}^{(\alpha,n)}}&=\left(\frac{(1-\alpha)[n-k]_q[n+k-1]_q+\alpha[n]_q [n-1]_q}{(1-\alpha)[n-k+j]_q[n+k-j-1]_q+\alpha[n]_q [n-1]_q}\right)
\prod_{m=k-j}^{k-1}\left(1-\frac{[m]_q}{[n]_q}\right).
\end{align*}
Obviously 
$$
\prod_{m=k-j}^{k-1}\left(1-\frac{[m]_q}{[n]_q}\right)<1.
$$
To complete the proof, note that
\begin{align*}
&[n-k]_q[n+k-1]_q\leq[n-k+j]_q[n+k-j-1]_q\\
\Leftrightarrow&(1-q^{n-k})(1-q^{n+k-1})\leq(1-q^{n-k+j})(1-q^{n+k-j-1})\\
\Leftrightarrow&(q^j-1)(q^{2k-j-1}-1)\geq 0
\end{align*}
for all $q>0$ and $j=1,2,\ldots,k-1.$
\end{proof}
\begin{remark}
It is worth mentioning that the numbers $\lambda^{(\alpha,n)}_{k,q}$ are the leading coefficients of $T_{n,q,\alpha}(t^k;x).$ That is, 
 $\lambda^{(\alpha,n)}_{k,q}=a(k,k)$, and hence 
$$
T_{n,q,\alpha}(t^k,x)=\lambda^{(\alpha,n)}_{k,q}x^k+P^{(\alpha,n)}_{k-1}(x)
$$
where $P^{(\alpha,n)}_{k-1}(x)$ is a polynomial of degree $k-1.$
\end{remark}

\section{Eigenvalues and eigenvectors of $T_{n,q,\alpha}$}

In  this part, the eigenvalues and the corresponding eigenvectors of $T_{n,q,\alpha}$ are found. The coefficients of the eigenvectors are given recursively. For some specific values of $n$, the eigenvalues are plotted.

\begin{lemma}
For all $q>0$ and $\alpha \in [0,1] $, the operator $T_{n,q,\alpha}$ has $n+1$ linearly independent monic eigenvectors 
$p_{k,q}^{(\alpha,n)}(x)$ of degree $k=0,1,\ldots,n$ corresponding to the eigenvalues $\lambda_{0,q}^{(\alpha,n)}=\lambda_{1,q}^{(\alpha,n)}=1$ and 
\begin{align*}
\lambda_{k,q}^{(\alpha,n)}=\frac{q^{\frac{k(k-1)}{2}}[n-2]_q!}{[n-k]_q![n]_q^k}(\left(1-\alpha)[n-k]_q[n-1+k]_q+\alpha[n]_q[n-1]_q\right)
\end{align*}
for $k=2,3,\ldots$. 
\end{lemma}
\begin{proof}
The proof is clear for $k=0$ and $k=1$. For $k=2,3,\ldots$, we have 
$$
T_{n,q,\alpha}(t^k;x)=\lambda_{k,q}^{(\alpha,n)}x^k+P_{k-1}^{(\alpha,n)}(x),
$$
where $P_{k-1}^{(\alpha,n)}(x)$ is a polynomial of degree $k-1$. 
Let 
$$p_{k}^{(\alpha,n)}(x)=x^k+\beta_{k-1}x^{k-1}+\ldots+\beta_1x$$ 
stand for the monic eigenvector of $T_{n,q,\alpha}$ corresponding to $\lambda_{k,q}^{(\alpha,n)},$ that is, 
$$
T_{n,q,\alpha}(p_k^{(\alpha,n)}(t);x)=\lambda_{k,q}^{(\alpha,n)}p_{k}^{(\alpha,n)}(x).
$$
Since $T_{n,q,\alpha}$ is linear, this equality becomes
\begin{align*}
T_{n,q,\alpha}(t^k;x)+\beta_{k-1}T_{n,q,\alpha}(t^{k-1};x)+\cdots+\beta_1T_{n,q,\alpha}(t;x)&=\lambda_{m}^{(\alpha,n)}(x^k+\beta_{k-1}x^{k-1}+\cdots+\beta_{1}x).
\end{align*}
Comparing the coefficients of $x^m,$ $m=1,2,\ldots, k-1$, we get the system in the unknowns $\beta_1,$ $\beta_2,$ $\ldots,$ $\beta_{k-1},$ whose coefficient matrix is
$$
\Lambda=\begin{pmatrix}
\lambda_{k,q}^{(\alpha,n)}-\lambda_{k-1,q}^{(\alpha,n)}& 0 & 0 & \cdots & 0\\
 * & \lambda_{k,q}^{(\alpha,n)}-\lambda_{k-2,q}^{(\alpha,n)}& 0 & \cdots &0\\
*&*&\lambda_{k,q}^{(\alpha,n)}-\lambda_{k-3,q}^{(\alpha,n)} & \ddots&0\\
*&* & * &\dots&\lambda_{k,q}^{(\alpha,n)}-\lambda_{1,q}^{(\alpha,n)}
\end{pmatrix}.
$$
Clearly    
$$
\det(\Lambda)=(\lambda_{k,q}^{(\alpha,n)}-\lambda_{k-1,q}^{(\alpha,n)})(\lambda_{k,q}^{(\alpha,n)}-\lambda_{k-2,q}^{(\alpha,n)})\cdots(\lambda_{k,q}^{(\alpha,n)}-\lambda_{1,q}^{(\alpha,n)})\neq 0
$$
by Lemma ~\ref{eigdifq}. Thus, there exist unique numbers $\beta_1,\ldots,\beta_{k-1}$ and hence $p_{k}^{(\alpha,n)}(x)$ is an eigenvector of $T_{n,q,\alpha}$ corresponding to the eigenvalue $\lambda_{k,q}^{(\alpha,n)}$.
\end{proof}

\begin{theorem}\label{thmeigq}
The monic eigenvector $p_{k,q}^{(\alpha,n)}(x)$ of $T_{n,q,\alpha}$ associated with
$\lambda_{k,q}^{(\alpha,n)}$ is a polynomial of degree $k$ given by
\begin{align*}
p_{k,q}^{(\alpha,n)}(x) = \sum_{j=0}^k c_{n,q}(j, k) x^j,
\end{align*}
where $p_{0,q}^{(\alpha,n)}(x)=1,$ $p_{1,q}^{(\alpha,n)}(x)=x$ and
\begin{align*}
c_{n,q}(k-j, k) =
\frac{1}{\lambda_{k,q}^{(\alpha,n)}-\lambda_{k-j,q}^{(\alpha,n)}}\,\sum_{i=0}^{j-1} c_{n,q}(k-i, k) a_{n,q}(k-j,k-i) 
\end{align*}
for $j=1,2,\ldots, k$, $k=2,3,\ldots$. 
\end{theorem}

\begin{proof} Let us write the eigenvector $p_{k,q}^{(\alpha,n)}(x)$ of $T_{n,q,\alpha}$ in the form
\begin{align*}
p_{k,q}^{(\alpha,n)}(x) = \sum_{r=0}^k c_{n,q}(r, k) x^r.
\end{align*} 
By the assumption that $c_{n,q}(k,k)=1$, the relation 
$$
T(p_{k,q}^{(\alpha,n)}; x)=\lambda_{k,q}^{(\alpha,n)}p_{k,q}^{(\alpha,n)}(x)
$$ 
implies
\begin{align*}
\lambda_{k,q}^{(\alpha,n)} \sum_{s=0}^k c_{n,q}(s, k) x^s&= \sum_{r=0}^k c_{n,q}(r, k) T(t^r; x) \\
&= \sum_{r=0}^k c_{n,q}(r, k) \sum_{i=0}^{r} a_{n,q}(i, r)x^i\\ 
&= \sum_{s=0}^k \sum_{r=s}^{k} c_{n,q}(r, k) a_{n,q}(s, r)x^s,
\end{align*}
which leads to
\begin{align}\label{b}
\lambda_{k,q}^{(\alpha,n)}c_{n,q}(s, k) =\sum_{r=s}^{k} c_{n,q}(r, k) a_{n,q}(s, r).
\end{align}
Substituting $s=k-j$ and $r=k-i$ on \eqref{b}, we obtain
\begin{align}\label{b.}
\lambda_{k,q}^{(\alpha,n)}c_{n,q}(k-j, k)&=\sum_{i=0}^{j} c_{n,q}(k-i, k) a_{n,q}(k-j, k-i)\nonumber\\
&=c_{n,q}(k-j, k) a_{n,q}(k-j, k-j)+\sum_{i=0}^{j-1} c_{n,q}(k-i, k) a_{n,q}(k-j, k-i).
\end{align}
%$$\lambda_{k,q}^{(\alpha,n)}c_{n,q}(k-j, k)-c_{n,q}(k-j, k) a_{n,q}(k-j, k-i)=
%\sum_{i=0}^{j-1} c_{n,q}(k-i, k) a_{n,q}(k-j, k-i).
%$$
%$$c_{n,q}(k-i, k)(\lambda_{k,q}^{(\alpha,n)}-\lambda_{k-j,q}^{(\alpha,n)})=
%\sum_{i=0}^{j-1} c_{n,q}(k-i, k) a_{n,q}(k-j, k-i).
%$$
It is seen from ~\eqref{b.} that
$$
c_{n,q}(k-j, k) = \frac{1}{\lambda_{k,q}^{(\alpha,n)}-\lambda_{k-j,q}^{(\alpha,n)}}
\sum_{i=0}^{j-1} c_{n,q}(k-i, k) a_{n,q}(k-j, k-i).
$$
The proof of Theorem ~\ref{thmeigq} is completed
\end{proof}

\begin{example}
As mentioned before $p_{0,q}^{(\alpha,n)}(x)=1$ and $p_{1,q}^{(\alpha,n)}(x)=x$ for all $n.$ Also, by the endpoint interpolation, one can easily derive that
$p_{2,q}^{(\alpha,n)}(x)=x^2-x$ for all $n.$ For $n=3$ and $k=3,$ one has
$$
p_{3,q}^{(\alpha,3)}(x)=x^3+a_2x^2+a_1x
$$
where
\begin{align*}
a_2&=-\frac{(1-\alpha)q^4+(2-\alpha)q^3+3q^2+(2\alpha+1)q+2}{(1-\alpha)q^4+q^3+2q^2+(1+\alpha)q+1}\\
a_1&=\frac{(1-\alpha)q^3+q^2+\alpha q+1}{(1-\alpha)q^4+q^3+2q^2+(1+\alpha)q+1}.
\end{align*}
The graph of $p_{3,q}^{(\alpha,3)}(x)$ for $\alpha=0.4$ and various values $q$ is plotted on Figure \ref{fig1}. 
Also, the graph of $p_{3,q}^{(\alpha,3)}(x)$ for fixed $q$ and variable $\alpha$ is depicted on Figure \ref{fig2}.  

\begin{figure}[!ht]
\begin{center}
{\includegraphics[width=0.8\textwidth]{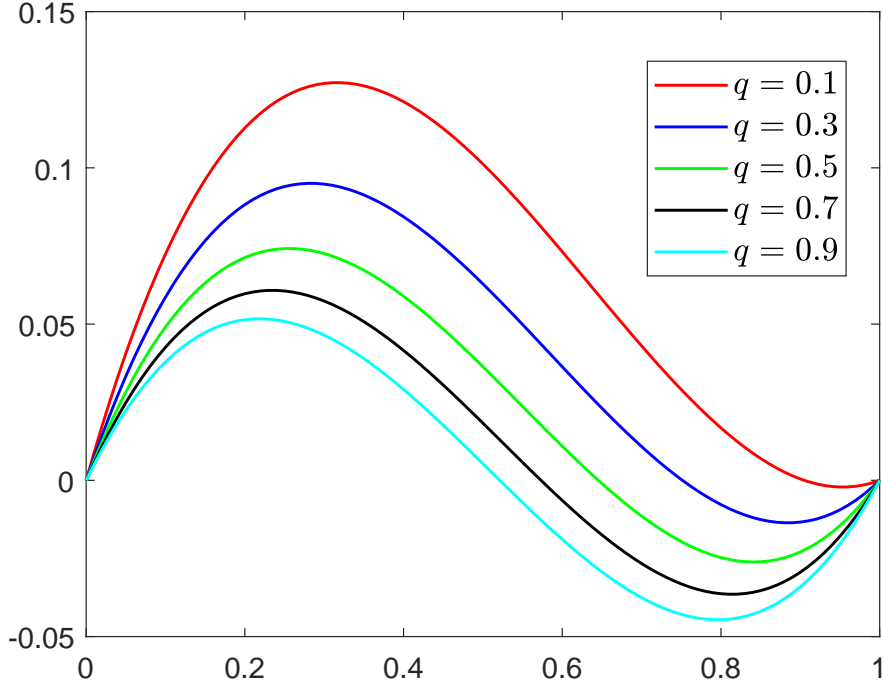}} 
\caption{The eigenvector $p_{3,q}^{(\alpha,3)}(x)$ for $\alpha=0.4$ and different values of $q.$} 
\label{fig1}
\end{center}\vspace{-4mm}
\end{figure}

\begin{figure}[!ht]
\begin{center}
{\includegraphics[width=0.8\textwidth]{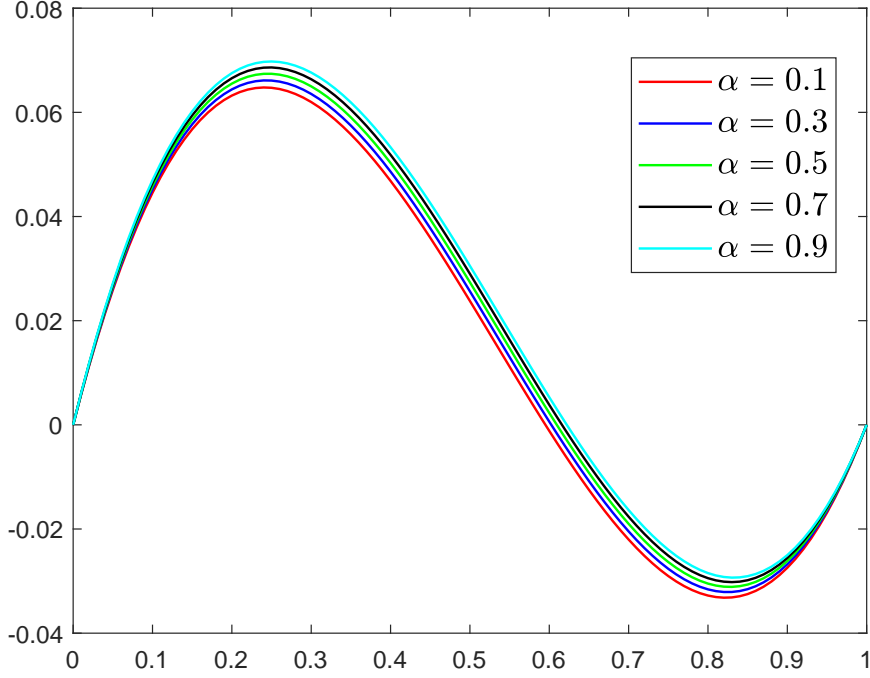}} 
\caption{The eigenvector $p_{3,q}^{(\alpha,3)}(x)$ for $q=0.6$ different values of $\alpha.$} 
\label{fig2}
\end{center}\vspace{-4mm}
\end{figure}

\end{example}

\subsection{The limit behavior of the eigenvalues and eigenvectors of $T_{n,q,\alpha}$}

\begin{lemma}
For $q\in(0, 1)$ and $k=2,3,\ldots,$ one has:
\begin{itemize}
\item [(i)] $\displaystyle \lim_{n\rightarrow\infty} a_{n,q}(r,k)=q^\frac{r(r-1)}{2}(1-q)^{k-r}S_q(k,r),$ $r=0,1,\ldots, k$
\end{itemize}
\begin{itemize}
\item [(ii)] $\displaystyle \lim_{n\rightarrow\infty} \lambda_{k,q}^{(\alpha,n)}=q^{k(k-1)/2}.$
\end{itemize}
\end{lemma}
\begin{proof}
The proof follows, immediately, from the fact that $\lim_{n\to\infty}[n]_q=1/(1-q),$ for all $q\in(0, 1).$ 
\end{proof}

\begin{theorem}
Let $0 < q < 1$ and $p_{k,q}^{(\alpha,n)}(x),$ $k=0,1,\ldots, n$ be the monic eigenvectors of $T_{n,q,\alpha}$ given in Theorem \ref{thmeigq}. Then 
\begin{eqnarray*}
\lim_{n\to\infty} c_{n,q}(j,k)=b_q(j,k)
\end{eqnarray*}
holds for $0\leq j \leq k,$ $k=0,1,\dots$, where
\begin{eqnarray*}
b_q(j,k)=\left\{\begin{array}{cl}
1, & j=k, \\
0, & j=k-1=0,\\
\displaystyle \sum_{i=j+1}^k \frac{(1-q)^{i-j}S_q(i,k)}{q^{(k-j)(k+j-1)/2}-1} \ b_q(i, k), & \textrm{otherwise}.
\end{array}
\right.
\end{eqnarray*}
In other words, when $q\in (0, 1),$ we have
\begin{align*}
\lim_{n\to\infty} p_{k,q}^{(\alpha,n)}(x)=p_k(x)=\sum_{j=0}^k b_q(j,k)x^j
\end{align*}
 uniformly on $[0,1].$
\end{theorem}
\begin{proof}
$p_0^{(n)}(x)=p_0(x)=1$ and $p_1^{(n)}(x)=p_1(x)=x,$ by Theorem \ref{thmeigq}. So it is enough to consider the case $k\geq2.$
Suppose that
$\lim_{n\to\infty} c_{n, q}(k-i, k)=b_q(k-i, k),
i=0,\ldots,j-1,$ where $0\leq j \leq k.$ 

One can easily see that 
\begin{equation*}
\lim_{n\to\infty}\frac{a_{n,q}(k-j,k-i)}{\lambda_{k,q}^{(\alpha,n)}-\lambda_{k-j,q}^{(\alpha,n)}}=\frac{(1-q)^{j-i}S_q(k-i,k-j)}{q^{j(2k-j-1)/2}-1}.
\end{equation*}
Then 
\begin{align*}
\lim_{n\to\infty} c_{n,q}(k-j, k)=\sum_{i=0}^{j-1}\frac{(1-q)^{j-i}S_q(k-i, k-j)}{q^{j(2k-j-1)/2}-1}\ c_{n,q}(k-i, k).
\end{align*}
Substituting  $j$ by $k-j$ and $i$ by $k-i,$ we obtain
\begin{equation*}
b_q(j,k)=\sum_{i=j+1}^{k}\frac{(1-q)^{i-j}S_q(i, k)}{q^{(k-j)(k+j-1)/2}-1}b_q(i, k)
\end{equation*}
which completes the proof.
\end{proof}

\begin{lemma} \label{lemcn}
For $q>1,$ one has:
\begin{itemize}
\item [(i)] $\displaystyle \lim_{n\rightarrow\infty} \lambda_{k,q}^{(\alpha,n)}=1.$
\item [(ii)] 
$
\displaystyle \lim_{n\to \infty} \frac{a_{n,q}(k-j,k-i)}{\lambda_{k,q}^{(\alpha,n)}-\lambda_{k-j,q}^{(\alpha,n)}}=0
$
for $i<j-1,$ $j=1,2,\ldots,k-1.$
\item [(iii)] 
$$
\lim_{n\to \infty}\frac{a_{n,q}(k-j,k-j+1)}{\lambda_{k,q}^{(\alpha,n)}-\lambda_{k-j,q}^{(\alpha,n)}}
=-\frac{S_q(k-j+1,k-j)+(1-\alpha)q^{j-k}[k-j]_q[k-j+1]_q}{[k-1]_q+[k-2]_q+\cdots+[k-j]_q}
$$
for $j=1,2,\ldots,k-1.$
\end{itemize}
\end{lemma}
\begin{proof} The proof of (i) is obvious. For the proof of (ii) and (iii), one can write 
\begin{multline*}
a_{n,q}(r,k)=\frac{q^{\frac{r(r-1)}{2}}[n]_q!}{[n]_q^k[n-r]_q!}\bigg\{(1-\alpha)\frac{[n-r]_q[n-r+1]_q}{[n]_q[n-1]_q}S_q(k+1,r+1)\\
-(1-\alpha)\frac{[n-r]_q}{[n]_q}[r+1]_qS_q(k,r+1)+\alpha S_q(k,r)\bigg\}
\end{multline*}
and
\begin{eqnarray*}
\lambda_{k,q}^{(\alpha,n)}=\frac{q^{\frac{k(k-1)}{2}}[n]_q!}{[n]_q^k[n-k]_q!}\bigg\{(1-\alpha)\frac{[n-k]_q[n+k-1]_q}{[n]_q[n-1]_q}+\alpha\bigg\}
\end{eqnarray*}
then
\begin{align*}
\frac{a_{n,q}(k-j,k-i)}{\lambda_{k,q}^{(n,\alpha)}-\lambda_{k-j,q}^{(n,\alpha)}}&=\frac{\frac{q^{\frac{(k-j)(k-j-1)}{2}}[n]_q!}{[n]_q^{k-i}[n-k+j]_q!}A_n}{\frac{q^{\frac{k(k-1)}{2}}[n]_q!}{[n]_q^k[n-k]_q!}B_n-\frac{q^{\frac{(k-j)(k-j-1)}{2}}[n]_q!}{[n]_q^{k-j}[n-k+j]_q!}C_n}\nonumber\\
&=\frac{[n]_q^iA_n}{[n-k+j]_q\cdots[n-k+1]_qq^{j(2k-j-1)/2}B_n-[n]_q^jC_n},
\end{align*}
where
\begin{align*}
A_n&=(1-\alpha)\frac{[n-k+j]_q[n+k-j-1]_q}{[n]_q[n-1]_q}S_q(k-i+1,k-j+1)\\
&\quad-(1-\alpha)\frac{[n-k+j]_q}{[n]_q}[k-j+1]_qS_q(k-i,k-j+1)+\alpha S_q(k-i,k-j),\\
B_n&=(1-\alpha)\frac{[n-k]_q[n+k-1]_q}{[n]_q[n-1]_q}+\alpha
\end{align*}
and 
\begin{align*}
C_n=(1-\alpha)\frac{[n-k+j]_q[n+k-j-1]_q}{[n]_q[n-1]_q}+\alpha.
\end{align*}
Using \eqref{prop1}, one can write
\begin{align*}
\frac{a_{n,q}(k-j,k-i)}{\lambda_{k,q}^{(\alpha,n)}-\lambda_{k-j,q}^{(\alpha,n)}}&=\frac{[n]_q^iA_n}{([n]_q-[k-j]_q)([n]_q-[k-j-1]_q)\cdots([n]_q-[k-1]_q)B_n-[n]_q^jC_n}\\
&=\frac{[n]_q^iA_n}{[n]_q^j(B_n-C_n)-[n]_q^{j-1}([k-1]_q+[k-2]_q+\cdots+[k-j]_q)B_n+\mathcal{O}([n]_q^{j-2})},
\end{align*}
Note that 
\begin{align*}
\lim_{n\rightarrow\infty}A_n&=(1-\alpha)\big(S_q(k-i+1,k-j+1)-q^{-k+j}[k-j+1]S_q(k-i,k-j+1)\big)\\
&\quad+\alpha S_q(k-i,k-j)\\
&=S_q(k-j+1,k-j)+(1-\alpha)q^{-k+j}[k-j]_q[k-j+1]_q
\end{align*}
where \eqref{str2} is used. Also,
$\lim_{n\rightarrow\infty}B_n=1$ and $\lim_{n\rightarrow\infty}C_n=1$. 
It is obvious now that if $i<j-1,$ then 
$$
\lim_{n\to \infty} \frac{a_{n,q}(k-j,k-i)}{\lambda_{k,q}^{(\alpha,n)}-\lambda_{k-j,q}^{(\alpha,n)}}=0
$$
which is the claim in (ii).
Moreover, if $i=j-1,$ then we have 
\begin{align*}
\lim_{n\to \infty} \frac{a_{n,q}(k-j,k-i)}{\lambda_{k,q}^{(\alpha,n)}-\lambda_{k-j,q}^{(\alpha,n)}}=-\frac{S_q(k-j+1,k-j)+(1-\alpha)q^{-k+j}[k-j]_q[k-j+1]_q}{[k-1]_q+[k-2]_q+\cdots+[k-j]_q}
\end{align*}
which completes the proof.
\end{proof}

\begin{theorem}
For $q > 1$ and $0\leq j \leq k,$ we have
\begin{eqnarray*}
\lim_{n\to\infty} c_{n,q}(j,k)=d_q(j,k), 
\end{eqnarray*}
where
\begin{eqnarray*}
d_q(0, 1)= 0, \quad d_q(j,k)=\prod_{i=1}^{k-j}-\frac{S_q(k-i+1, k-i)+(1-\alpha)q^{-k+i}[k-i]_q[k-i+1]_q}{[k-1]_q+\cdots+[k-i]_q}
\end{eqnarray*}
\end{theorem}

\begin{proof}
For $k=0$ and $k=1,$ there is nothing to prove. Assume that $k\geq 2,$ and use strong induction on $j.$ Since $c_{n,q}(k,k)=1,$ the statement is true for $j=k.$ 
Assume that $\lim_{n\to\infty} c_{n,q}(k-i,k)=d_q(k-i,k)$ for $i=0,1,\ldots,j-1.$ Then using Lemma \ref{lemcn} (i) and (ii), we obtain
\begin{align*}
\lim_{n\to\infty} c_{n,q}(k-j,k)&=\sum_{i=0}^{j-1}d_q(k-i,k)\lim_{n\to\infty}\frac{a_{n,q}(k-j,k-i)}{\lambda_{k,q}^{(\alpha,n)}-\lambda_{k,q}^{(\alpha,n)}}\\
&=-\frac{S_q(k-j+1, k-j)+(1-\alpha)q^{-k+j}[k-j]_q[k-j+1]_q}{[k-1]_q+\cdots+[k-i]_q}d_q(k-j+1,k)
\end{align*}
From the fact that 
$$d_q(k-j+1,k)=\prod_{i=1}^{j-1}-\frac{S_q(k-i+1, k-i)+(1-\alpha)q^{-k+i}[k-i]_q[k-i+1]_q}{[k-1]_q+\cdots+[k-i]_q}$$
then we get 
\begin{eqnarray*}
d_q(j,k)=\prod_{i=1}^{k-j}-\frac{S_q(k-i+1, k-i)+(1-\alpha)q^{-k+i}[k-i]_q[k-i+1]_q}{[k-1]_q+\cdots+[k-i]_q}
\end{eqnarray*}
which completes the the induction.
\end{proof}

\begin{remark} Note that when $q>1,$ the limiting coefficients depend on $\alpha$ unlike the case $0<q<1.$ 
\end{remark}

\end{document}